\documentclass[12pt]{article}
\usepackage{amsfonts}
\usepackage{amsmath}
\usepackage{amssymb}
\usepackage{amsthm}
\usepackage{hyperref}

\hypersetup{
  colorlinks,
  urlcolor=black,
  linkcolor=black,
  anchorcolor=black,
  citecolor=black,
}

\title{The final size of the $C_4$-free process}

\author{Michael E. Picollelli\footnote{
\small Department of Electrical \& Computer Engineering, University of Delaware, Newark, DE, USA.  E-mail: \texttt{mpicolle@udel.edu}}}
\date{}

\newcommand{\ep}{\varepsilon}
\newcommand{\bin}[2]{\binom{#1}{#2}}
\newcommand{\pr}[1]{\mbox{Pr}\left(#1\right)}
\newcommand{\ev}[1]{\mathbb{E}\left(#1\right)}
\newcommand{\lp}{\left(}
\newcommand{\rp}{\right)}
\newtheorem{theorem}{Theorem}

\newtheorem{lemma}{Lemma}
\newtheorem{corollary}{Corollary}
\newtheorem{claim}{Claim}

\newcommand{\ca}[1]{{\cal #1}}
\newcommand{\es}{\text{e}}

\begin{document}

\maketitle

\begin{abstract}
We consider the following random graph process: starting with $n$ isolated vertices, add edges uniformly at random provided no such edge creates a copy of $C_4$.  We show that, with probability tending to $1$ as $n \to \infty$, the final graph produced by this process has maximum degree $O((n \log n)^{1/3})$ and consequently size $O(n^{4/3}\log(n)^{1/3})$, which are sharp up to constants.  This confirms conjectures of Bohman and Keevash and of Osthus and Taraz, and improves upon previous bounds due to Bollob\'as and Riordan and Osthus and Taraz.
\end{abstract}

\section{Introduction}\label{sec:intro}
The $H$-free process, where $H$ is a fixed graph, is the random graph process which begins with a graph $G(0)$ on $n$ isolated vertices.  The graph $G(i)$ is then formed by adding an edge $e_i$ selected uniformly at random from the pairs which neither form edges of $G(i-1)$ nor create a copy of $H$ in $G(i-1) + e_i$.  The process terminates with a maximal $H$-free graph $G(M)$ with $M=M(H)$ edges.

Erd{\H o}s, Suen and Winkler \cite{ESW} suggested this process as a natural probability distribution on maximal $H$-free graphs, and asked for the typical properties of $G(M(H))$, such as size and independence number.  They considered the odd-cycle-free and triangle-free processes, establishing that the former terminates with $\Theta(n^2)$ edges with high probability\footnote{We say a sequence of events $A_n$ occurs \textbf{with high probability}, or simply \textbf{w.h.p.}, if $\lim_{n \to \infty} \pr{A_n} = 1$.}, and that, for some positive constants $c_1,c_2,c_3$,  w.h.p. $c_1n^{3/2} \le M(K_3) \le c_2n^{3/2}(\log n)$ and $\alpha(G(M(K_3))) \le c_3 \sqrt{n}( \log n)$.  These bounds were improved by Spencer \cite{Sp2}, who further conjectured that w.h.p. $M(K_3) = \Theta(n^{3/2} \sqrt{\log n})$.  (We mention that the earliest result on an $H$-free process is due to Ruci\'nski and Wormald \cite{RW}, who that the maximum-degree $d$ process terminates in a graph with $\lfloor nd/2 \rfloor$ edges with high probability - here $H$ is the star graph $K_{1,d+1}$.)

More general $H$-free processes, where $H$ satisfies an additional density condition, were first studied by Bollob\'as and Riordan \cite{BR} and by Osthus and Taraz \cite{OT} independently.  For the remainder of this section, the bounds we mention are assumed to hold w.h.p. unless stated otherwise.  We say a graph $H$ is \textbf{$2$-balanced} if $e(H) \ge 3$, $v(H)\ge 3$, and  \[\frac{e(H)-1}{v(H)-2} \ge \frac{e(F)-1}{v(F)-1}\] for all proper subgraphs $F$ of $H$ with $v(F) \ge 3$, and \textbf{strictly} $2$-balanced if the inequality is sharp for all such $F$; examples of such graphs include cycles, complete graphs, and complete bipartite graphs $K_{r,r}$, $r \ge 2$.  For this class of graphs, Bollob\'as and Riordan established general lower bounds on $M(H)$, and upper bounds for $H \in \{C_4,K_4\}$ that match to within a logarithmic factor.  Osthus and Taraz then gave upper bounds for all strictly $2$-balanced $H$ that match to within a logarithmic factor.  For $H=C_4$, the results of Bollob\'as and Riordan yield $M(C_4) = \Omega(n^{4/3})$ and $M(C_4) = O(n^{4/3} (\log n)^3)$; Osthus and Taraz's results improve the upper bound to $M(C_4) = O(n^{4/3} \log n)$, and they further conjectured that the average degree of the $C_l$-free process is $O((n \log n)^{1/(l-1)})$ for all $l \ge 3$.  Evidence that the lower bound was not sharp came from Wolfovitz \cite{Wolf} who improved the lower bound on $\ev{M(H)}$ for regular strictly $2$-balanced $H$ by a factor of $(\log \log n)^{1/(e(H)-1)}$.

Finally, through an application of the differential equations method (for the general method and examples, see \cite{W}), Bohman \cite{B} showed \(M(K_3) = \Theta(n^{3/2}\sqrt{\log n})\), confirming Spencer's conjecture, and produced an improvement on the lower bound for $M(K_4)$.  Subsequent work by Bohman and Keevash \cite{BK} established new lower bounds on $M(H)$ for all strictly $2$-balanced $H$ by producing lower bounds on the minimum degree of $G(M(H))$, and they conjectured that the likely maximum degree of $G(M(H))$ is at most a constant multiple of their lower bound.  For the $C_4$-free process, their bound on the minimum degree is $\Omega((n \log n)^{1/3})$, yielding $M(C_4) = \Omega(n^{4/3} (\log n)^{1/3})$.

We mention some motivation for studying the $H$-free process comes in part from its connection to two classical areas of extremal combinatorics, Ramsey theory and Tur\'an theory.  Bounds on the independence number of $G(M(K_3))$ found in \cite{ESW} and \cite{Sp2} led to the best lower bounds on $R(3,t)$ known at the time, and Bohman's analysis \cite{B} produced an improvement that matched Kim's celebrated lower bound \cite{K}.  The analysis in \cite{B} and \cite{BK} has also led to the best current lower bounds for the Ramsey numbers $R(s,t)$, with $s \ge 4$ fixed and $t$ large, and the cycle-complete Ramsey numbers $R(C_l,K_t)$ for $l \ge 4$ fixed and $t$ large.  The results for the $K_{r,r}$-free process in \cite{Wolf} and \cite{BK} have resulted in improvements on the best known lower bounds for the Tur\'an numbers $ex(n,K_{r,r})$ for $r \ge 5$.

However, the process has also become a subject of recent interest on its own, in part for aspects of Bohman and Keevash's analysis of the strictly $2$-balanced case that suggest the graph $G(i)$ produced by the process resembles the random graph $G(n,i)$, chosen uniformly at random from all $i$-edge graphs on $n$ vertices, with the exception that it contains no copies of $H$.  To establish their lower bound, they show that a wide range of subgraph extension variables, including the degree of a vertex and the number of copies of a given $H$-free graph $F$, take roughly the same values in $G(i)$ as in $G(n,i)$, for $i$ up to a small multiple of $n^{2-(v(H)-2)/(e(H)-1)}(\log n)^{1/(e(H)-1)}$.  (Similar results on subgraph counts in the $K_3$-free process were obtained by Wolfovitz \cite{Wolf3}.)  In fact, the lower bound and conjectured upper bounds on $M(H)$ in \cite{BK} correspond (within constant factors) to the threshold for the random graph $G(n,i)$ to have the property that the addition of \textit{any} new edge creates a copy of $H$, provided $H$ is strictly $2$-balanced (see \cite{Sp1}).  It is also known (see \cite{GM} and \cite{WA1}) that sufficiently dense subgraphs are unlikely to appear in the final graph $G(M(H))$.

Very recently, Warnke \cite{WA2} and Wolfovitz \cite{Wolf2} have independently given upper bounds on $M(K_4)$ that match Bohman's lower bound to within a constant factor. The author \cite{P} has also established similar bounds for the case where $H$ is the diamond graph, formed by removing an edge from $K_4$. (The diamond graph is $2$-balanced but not strictly so.)  Along with $K_3$, these are the only $2$-balanced graphs containing a cycle for which such bounds on $M(H)$ are currently known. Our aim is to add $C_4$ to this list through the next result.

\begin{theorem}\label{thm:C4upperbd}
There exists $\kappa>0$ such that $\Delta(G(M(C_4))) \le \kappa (n \log(n))^{1/3}$ with high probability.
\end{theorem}
\noindent This confirms the mentioned conjectures of Osthus and Taraz and of Bohman and Keevash for the $C_4$-free process.  Combined with the lower bound given in \cite{BK}, this has the following immediate corollary.
\begin{corollary}\label{cor:finalsize}
With high probability, $M(C_4) = \Theta(n^{4/3} (\log n)^{1/3})$.
\end{corollary}

\noindent From an upper bound established in \cite{BK}, as well as known bounds on the independence number of $C_4$-free graphs with bounded maximum degree, we arrive at the next result easily.

\begin{corollary}\label{cor:indepno}
With high probability, $\alpha(G(M(C_4))) = \Theta( (n \log n)^{2/3})$.
\end{corollary}

\noindent An immediate consequence of this second corollary is that a typical graph produced by the $C_4$-free process will not essentially improve the lower bound on $R(C_4,K_t)$ given in \cite{BK}.

To establish our bound, we use a fairly simple observation: suppose we fix a vertex $v$ and a step $i \le M=M(C_4)$.  If $x$ and $y$ are neighbors of $v$ in $G(M)$ but are nonadjacent to $v$ in $G(i)$, then $x$ and $y$ have no common neighbors in $G(i)$.  We can therefore establish an upper bound on $\Delta(G(M))$ of the form $\Delta(G(i)) + k$ by showing that every set of $k$ vertices contains two which share a neighbor in $G(i)$.  Thus, to prove Theorem \ref{thm:C4upperbd}, we simply need to make appropriate choices of $i$ and $k$.

The remainder of this paper is organized as follows: in the next section we discuss the $C_4$-free process specifically, including relevant results from \cite{BK}, and in Section \ref{sec:upperbd} we introduce our main technical lemma (Lemma \ref{lem:trajectory}) and prove Theorem \ref{thm:C4upperbd} and Corollary \ref{cor:indepno}.  Section \ref{sec:prelim} will cover a few preliminary results for our proof of Lemma \ref{lem:trajectory}, including a lemma from \cite{BK} which forms the basis for our differential equations method application, and the proof of Lemma \ref{lem:trajectory} will follow in Section \ref{sec:lemma1}.

\section{The $C_4$-free process}\label{sec:c4freeproc}

\subsection{Definitions and notation}\label{sec:defandnot}

We let $[n]=\{1,\ldots,n\}$ be the vertex set of the process, and $G(i)$ the graph given by the first $i$ edges selected by the process.  $G(i)$ naturally partitions $\bin{[n]}{2}$ into three sets, $E(i)$, $O(i)$, and $C(i)$.  $E(i)$ is simply the edge set of the process.  For a pair $uv \notin E(i)$, we say $uv$ is \textbf{open}, and $uv \in O(i)$, if the graph $G(i) + uv$ is $C_4$-free.  Otherwise, we say $uv$ is \textbf{closed} and $uv \in C(i)$.  For $v \in [n]$, we let $N_i(v)$ and $d_i(v)$ denote the neighborhood and degree, respectively, of $v$ in $G(i)$.

For $i \ge 0$ and a pair of vertices $uv \in \bin{[n]}{2} \setminus E(i)$, we define $C_{uv}(i)$ to be the set of pairs $wz \in O(i)$ such that $G(i) + uv + wz$ contains a copy of $C_4$ that uses both $uv$ and $wz$ as edges.  Equivalently, $C_{uv}(i)$ is the collection of open pairs which, if added as edges, would create a path of length three between $u$ and $v$. We mention that, in \cite{BK}, $C_{uv}$ is defined as the set of ordered pairs; we will work exclusively with unordered pairs.

We introduce a continuous time variable $t$, and relate it to the process by setting $t = t(i) = i/n^{4/3}$.  We fix constants $\mu,\ep, V, W$, which satisfy \[0 < \mu \ll \ep \ll \frac{1}{W} \ll \frac{1}{V} \ll \frac{1}{4}.\]  (The notation $0 < a \ll b$ means there is an increasing function $f(x)$ so the arguments which follow are valid for $0 < a < f(b)$.)  Given these constants, we define
\begin{equation}\label{eq:defpmtmax}
p = n^{-2/3}, \ \ \ \ \ m = \mu (\log n)^{1/3} \cdot n^{4/3}, \ \ \text{ and }\ \ \ t_{max} = \mu (\log n)^{1/3}.
\end{equation}
We further define functions $q(t), c(t), P(t), e(t)$ as well as parameters $s=s(n)$ and $s_e=s_e(n)$ as follows:
\begin{align}
\label{eq:qcdef} q(t) &= \exp(-8t^3), &    c(t) &= 24t^2\exp(-8t^3),\\
\label{eq:Pedef} P(t) &= W(t^3 + t), &  e(t) &= e^{P(t)}-1,\\
\label{eq:ssedef} s(n) &= n^2p = n^{4/3},\ \ \ \ \ \ \ \ \ \text{ and }&   s_e(n) &= n^{1/8-\ep}.
\end{align}

We assume that $\ep$ and $\mu$ are chosen sufficiently small that $e(t)$ and $q(t)^{-1}$ are at most $n^{\ep}$ for $0 \le t \le t_{max}$, and $s_e = n^{1/8-\ep} \gg n^{\ep}$, so $e(t)/s_e = o(1)$ (uniformly with respect to $n$) for $0 \le t \le t_{max}$.  We will discuss additional bounds on $\mu,\ep,V$ and $W$ further in Section \ref{sec:ineq}.

\subsection{The lower bound - results of Bohman and Keevash}\label{sec:lowerbd}

Bohman and Keevash \cite{BK} established their lower bound on the $H$-free process by showing that certain random variables are tightly concentrated throughout the initial $m$ steps.  As we do not require the full strength of their results, we summarize the relevant consequences for the $C_4$-free process in the following theorems.

\begin{theorem}[Bohman and Keevash, \cite{BK}]\label{thm:evoC4}
Let $\ca{T}_{i^*}$ denote the event that the following hold for $0 \le i \le i^*$:
\begin{itemize}

\item[1.] \begin{equation}\label{eq:Q}
Q(i) = \lp 1 \pm \frac{e(t)}{s_e} \rp \lp q(t) \pm \frac{1}{s_e} \rp \frac{n^2}{2},\\
\end{equation}

\item[2.] For all $v \in [n]$,
\begin{equation}\label{eq:regular}
d_i(v) = \lp 1 \pm \frac{e(t)}{s_e} \rp \lp 2t \pm \frac{1}{s_e}\rp np,
\end{equation} and so $\Delta(G(i)) \le 4t_{max}np$.

\item[3.] For all $uv \in O(i) \cup C(i)$,
\begin{equation}\label{eq:Cuv}
|C_{uv}(i)| = \lp 1 \pm \frac{e(t)}{s_e} \rp \lp 24t^2 q(t) \pm \frac{12}{s_e} \rp \frac{p^{-1}}{2},
\end{equation}
and for all distinct $uv,u'v' \in O(i)$,
\begin{equation}\label{eq:destfidel}
|C_{uv}(i) \cap C_{u'v'}(i)| \le n^{-1/4}p^{-1}
\end{equation}

\end{itemize}
Then $\ca{T}_m$ holds with high probability.
\end{theorem}

\begin{theorem}[Bohman and Keevash, \cite{BK}]\label{thm:indnoub}
With high probability, $\alpha(G(m)) \le 3\mu^{-1}(n \log n)^{2/3}$.
\end{theorem}

Recalling that we may choose $\ep$ and $\mu$ so that $e(t)/s_e = o(1)$ and $q(t) \ge n^{\ep} \gg 1/s_e$, $\ca{T}_m$ implies $Q(m)>0$ and consequently the lower bound $M(C_4) \ge \mu n^{4/3}(\log n)^{1/3}$ holds with high probability.

Equation \eqref{eq:Q} follows immediately from Theorem 1.4 of \cite{BK}.  Equation \eqref{eq:regular} follows similarly, while the bound on $\Delta(G(i))$ follows from bounding $e(t)/s_e$ above by $1/3$ and $1/s_e$ above by $t_{max}$.  Equations \eqref{eq:Cuv} and \eqref{eq:destfidel} follow from Corollary 6.2 and Lemma 8.4 of \cite{BK}, respectively.  We mention that the phrasing of Lemma 8.4 suggests that $uv$ and $u'v'$ are fixed.  However, as Lemma 8.4 is shown to be a consequence of a constant (depending on $H$) number of applications of Lemma 5.2, which has exponentially small failure probability (conditioned on their event $\ca{G}_m$), \eqref{eq:destfidel} follows for all such pairs and steps $i$, $0 \le i \le m$, with high probability.  Finally, Theorem \ref{thm:indnoub} follows from the proof of Theorem 1.9 of \cite{BK} (specifically, Lemmas 11.3 and 12.1).

\subsection{The upper bound - proofs of Theorem \ref{thm:C4upperbd} and Corollary \ref{cor:indepno}}\label{sec:upperbd}

Let $\beta>0$ be a fixed constant satisfying
\begin{equation}\label{eq:betabd}
\beta > \frac{4}{\mu^2},
\end{equation}
and define
\begin{equation}\label{eq:kdef}
k = \beta\cdot (n \log n)^{1/3}.
\end{equation}
Our aim is to show that, in $G(m)$, every $k$-element subset of $[n]$ contains two vertices that share a neighbor.  We say a $k$-set $K$ is \textbf{covered} in $G(i)$ if a common neighbor exists in $G(i)$ for some pair of vertices in $K$; $K$ is \textbf{uncovered} otherwise.  As previously mentioned, for any $v \in [n]$, no pair of vertices in $N_{M(C_4)}(v)\setminus N_m(v)$ can have a common neighbor in $G(m)$ - if every $k$-set is covered in $G(m)$, this yields $d_{M(C_4)}(v) \le d_m(v) + k = O( (n \log n)^{1/3})$, the desired bound.

If a set $K$ is covered in $G(i)$, then there exists a triple of vertices $u,v,w$ such that $uv \in \bin{K}{2}$ and $uw,vw \in E(i)$.  We note that the order of $u$ and $v$ is not essential, and that we expect that it is likely (but not necessary!) that the common neighbor $w$ does not lie in $K$.  We therefore restrict ourselves to considering certain subsets of $\bin{K}{2} \times ([n] \setminus K)$.  We will write elements of $\bin{K}{2} \times([n]\setminus K)$ as $(uv,w)$ but will refer to them as triples to avoid confusion with elements of $\bin{[n]}{2}$.
We will also identify each such triple $(uv,w)$ with the subset $\{uw,vw\}$ of $\bin{[n]}{2}$.

We introduce the following definitions: given $K \in \bin{[n]}{k}$ and $0 \le i \le m$, we define
\begin{align*}
X_K(i)& = \left\{(uv,w) \in \bin{K}{2} \times ([n]\setminus K) : uw,vw \in O(i)\right\}, \text{ and }\\
Y_K(i)& = \left\{(uv,w) \in \bin{K}{2} \times ([n]\setminus K) : |\{uw,vw\} \cap O(i)| = |\{uw,vw\} \cap E(i)|= 1\right\}.
\end{align*}

We call triples in $X_K(i)$ \textbf{open with respect to} $K$, and triples in $Y_K(i)$ \textbf{partial with respect to} $K$.  We note that if $(uv,w) \in Y_K(i)$, where, without loss of generality, $uw \in O(i)$, then if we select $e_{i+1} = uw$, $K$ is covered in all steps $i' \ge i+1$.  Equally important, if $K$ is uncovered in $G(i)$, then every pair in $O(i)$ lies in at most one triple $(uv,w) \in Y_K(i)$ (viewed as $\{uw,vw\}$), as otherwise $w$ has at least two neighbors in $K$.

Intuitively, the probability that a given pair of vertices $xy$ is open at time $t=t(i)$ is $\approx q(t)$, while the probability that a pair is an edge is $\approx 2tp$.  It is then reasonable to suspect that
\[|X_K(i)| \approx \bin{k}{2}\cdot (n-k) \cdot q(t)^2 \approx \frac{q(t)^2}{2} \cdot k^2n, \text{ and } \] \[ |Y_K(i)| \approx 2\bin{k}{2}(n-k)(q(t))(2tp) \approx 2tq(t)\cdot k^2np.\]
The following lemma shows that these estimates are correct for uncovered $K$.

\begin{lemma}\label{lem:trajectory}
With high probability, for all $i$, $0 \le i \le m$, and $K \in \bin{[n]}{k}$, if $K$ is uncovered in $G(i)$ then
\begin{align*}
|X_K(i)| &= \lp 1 \pm \frac{e(t)}{n^{3\ep}}\rp \lp \frac{q(t)^2}{2} \pm \frac{1}{n^{3\ep}} \rp k^2n, \text{ and }\\
|Y_K(i)| &= \lp 1 \pm \frac{e(t)}{n^{3\ep}}\rp \lp 2tq(t) \pm \frac{1}{n^{3\ep}} \rp k^2np.
\end{align*}
\end{lemma}
\noindent  Via an argument similar to that used to bound the independence number of the $K_r$-free and $C_r$-free processes in \cite{B} and \cite{BK}, we next show how Lemma \ref{lem:trajectory} implies Theorem \ref{thm:C4upperbd}.

\begin{proof}[Proof of Theorem \ref{thm:C4upperbd}]  We may assume the conclusions of Theorem \ref{thm:evoC4} and Lemma \ref{lem:trajectory} hold, as their failure probability is $o(1)$.  We also assume that $\mu,\ep,V$ and $W$ are chosen sufficiently small that $q(t)^{-1}$ and $e(t)$ are at most $n^{\ep}$ on $[0,t_{max}]$, and $s_e \ge n^{3\ep}$.
By Theorem \ref{thm:evoC4}, with high probability, $\Delta(G(m)) \le 4\mu (n \log n)^{1/3}$.  Letting $\kappa = 4\mu + \beta$, we establish the bound in Theorem \ref{thm:C4upperbd} by showing that w.h.p. every $k$-set is covered in $G(m)$.

Given an uncovered $K$ at step $i$, as $e_{i+1}$ is chosen uniformly at random from $Q(i)$ open pairs, and as each partial triple $(uv,w) \in Y_K(i)$ contains a unique open pair, the probability $K$ remains uncovered in $G(i+1)$ is at most $1 - \frac{|Y_K(i)|}{Q(i)}$.  We restrict our attention to bounding the probability that some $k$-set $K$ remains uncovered for all steps $i$, $m/2 \le i \le m$.  For $n$ sufficiently large, $m/2 \ge n^{4/3}$, so in this range of $i$ we may assume $t=t(i) \ge 1$.  Thus, if $K$ is uncovered in $G(i)$ with $m/2 \le i \le m$, then as $s_e \ge n^{3\ep}$, $e(t)/n^{3\ep} \le 1/3$ and $1/n^{3\ep} \le q(t)/2 \le tq(t)/2$
for $t_{max}/2 \le t \le t_{\max}$, from Theorem \ref{thm:evoC4} and Lemma \ref{lem:trajectory} we have
\[\frac{|Y_K(i)|}{Q(i)} = \frac{(1 \pm e(t)/n^{3\ep})(2tq(t) \pm 1/n^{3\ep})k^2np}{(1 \pm e(t)/s_e)(q(t) \pm 1/s_e)n^2/2} \ge \frac{2/3\cdot(3/2)tq(t)k^2np}{4/3\cdot (3/2) q(t)n^2/2}
\ge \frac{t_{\max}k^2p}{2n}.\]

Therefore, the probability that a $k$-set $K$ exists which remains uncovered for all $i$, $m/2 \le i \le m$ is at most
\begin{eqnarray*}
\bin{n}{k} \lp 1-\frac{t_{max}k^2p}{2n}\rp ^{m/2} &\le& n^k\exp\lp -\frac{t_{\max}k^2pm}{4n}\rp\\
&=& n^k\exp\lp -\frac{\mu (\log n)^{1/3} (\beta (n\log n)^{1/3})^2 \cdot n^{-2/3}\cdot \mu n^{4/3}(\log(n))^{1/3}}{4n}\rp\\
&=& n^k\exp\lp -\frac{\mu^2\beta^2n^{1/3}(\log n)^{4/3}}{4}\rp.
\end{eqnarray*} As $n^k = \exp(\beta n^{1/3} (\log n)^{4/3})$, this probability is $o(1)$ provided $\mu^2\beta^2/4 > \beta$, i.e. \eqref{eq:betabd} holds.
\end{proof}

We note again that Corollary \ref{cor:finalsize} follows immediately, so we turn to the proof of Corollary \ref{cor:indepno}.  The upper bound on $\alpha(G(M(C_4)))$ follows from Theorem \ref{thm:indnoub}, as $\alpha(G(M(C_4))) \le \alpha(G(m))$.  For the lower bound, we apply a lemma from \cite{Bo} bounding the independence number of graphs with few triangles.  (Similar bounds are known for a wider class of $H$-free graphs - see \cite{AKS}.)

\begin{lemma}[\cite{Bo}, Lemma 12.16 (ii)]\label{lem:indnolem}
Let $G$ be a graph on $n$ vertices with average degree at most \ $d$ and at most $h$ triangles.  Then
\[\alpha(G) \ge \frac{1}{10}\cdot \frac{n}{d} \lp \log d - \frac{1}{2} \log\lp\frac{h}{n}\rp\rp.\]
\end{lemma}

\begin{proof}[Proof of Corollary \ref{cor:indepno}]
Letting $M = M(C_4)$, by Theorem \ref{thm:C4upperbd}, with high probability the average degree of $G(M)$ is at most $\hat{\kappa}=\kappa(n \log n)^{1/3}$, where $\kappa>0$ is a fixed constant.  As $G(M)$ is $C_4$-free, each edge lies on at most one triangle, so $G(M)$ has at most $M/3$ triangles.  Taking $d = \hat{\kappa}$ and $h = \hat{\kappa}n$ in Lemma \ref{lem:indnolem}, and observing that $\log(\hat{\kappa}) \ge \log(n)/3$ for $n$ sufficiently large, we have \[\alpha(G(M)) \ge \frac{1}{10}\cdot \frac{n}{\hat{\kappa}}\lp \frac{1}{2} \log \hat{\kappa}\rp \ge \frac{1}{60\kappa}(n \log n)^{2/3}.\]
\end{proof}

\section{Preliminaries}\label{sec:prelim}

We use the notation ``$\pm$" in two distinct ways throughout this paper.  The notation $a \pm b$ will be taken to mean the interval $\{a + xb : -1 \le x \le 1\}$; distinct instances of $\pm$ used this way in the same expression will be treated independently, i.e. $(a \pm b)(c \pm d)$ will be taken to mean $\{(b + x_1c)(d + x_2e) : -1 \le x_1,x_2 \le 1\}$.  We will also write $a = b \pm c$ instead of $a \in b \pm c$.

For a sequence of random variables $A(1),A(2),\ldots,$, we will use $A^{\pm}$ to denote pairs of sequences of nonnegative random variables $A^+(1),A^+(2),\ldots$ and $A^-(1),A^-(2),\ldots$, such that \[A(i+1)-A(i) = A^+(i) - A^-(i).\]  Similarly, for a differentiable function $f(t)$, we will use $f^+$ and $f^-$ to denote the positive and negative parts of $f'(t)$.

\subsection{A density claim}\label{sec:3paths}

An important part of our argument will be showing that the maximum one-step change in the variables we track is sufficiently bounded.  This will turn out to be straightforward for the number of open triples with respect to a given $K$, but to establish effective bounds on the maximum one-step decrease in the number of partial triples, we will appeal to a simple bound on the number of paths of length three between any two vertices in $G(i)$.  We mention that in the binomial random graph $G(n,p)$, the expected number of such paths is $(n-2)(n-3)p^2 \approx 1$, so the upper bound we establish below of $n^{1/4}$ is reasonable to expect.  We mention that this bound is by no means optimal, but more than suffices for our arguments.

\begin{lemma}\label{lem:3path}
Let $\ca{P}_i$ be the event that, in $G(i)$, for every pair of distinct vertices $u$ and $v$, there are at most $n^{1/4}$ paths of length $3$ between them.  Then, conditioned on $\ca{T}_m$, $\ca{P}_m$ holds with high probability.
\end{lemma}

\noindent To prove Lemma \ref{lem:3path} we first establish a simple claim.

\begin{claim}\label{clm:disjoint3paths}
Given distinct vertices $u$ and $v$ in a $C_4$-free graph $G$,
any two paths of length $3$ between $u$ and $v$ are edge-disjoint.
\end{claim}

\begin{proof}
Suppose $(u,x,y,v)$, $(u,x',y',v)$ are distinct paths from $u$ to $v$ in $G$ that share an edge.  If the shared edge is $ux$ (so $x=x'$), then the vertices $x,y,v,y'$ form a $C_4$, a contradiction; similarly if the shared edge is $vy$.  If the shared edge is $xy$, then, as the two paths are distinct, $x'=y, y'=x$, and the vertices $u,y,v,x$ form a $C_4$, again a contradiction.
\end{proof}

\begin{proof}[Proof of Lemma \ref{lem:3path}]
We fix vertices $u$ and $v$ and bound the probability that there are $n^{1/4}$ such paths between them in $G(m)$.  On any such path between $u$ and $v$, for the last of the three edges added, $e=e_i$, we must have $e_i \in C_{uv}(i-1)$ by definition.  Conditioned on $\ca{T}_m$, for $i=1,\ldots,m$, the probability that $e_i \in C_{uv}(i-1)$ is at most $|C_{uv}(i-1)|/Q(i) \le n^{\ep}p^{-1}/n^{2-\ep} = n^{2\ep}/n^{4/3}$.  By Claim \ref{clm:disjoint3paths}, as any two such paths are edge-disjoint, it suffices to bound the probability that $e_i \in C_{uv}(i-1)$ for $n^{1/4}$ steps $i$: noting $m \le n^{\ep}\cdot n^{4/3}$ for $n$ sufficiently large, this is at most \[\bin{m}{n^{1/4}}\cdot \lp \frac{n^{2\ep}}{n^{4/3}}\rp^{n^{1/4}} \le \lp \frac{men^{2\ep}}{n^{1/4}n^{4/3}} \rp^{n^{1/4}}\le \lp \frac{en^{3\ep}}{n^{1/4}} \rp^{n^{1/4}} \le \exp(-n^{1/4}).\]
The result then follows from a union bound over the $\bin{n}{2}$ choices of $u$ and $v$.
\end{proof}

\noindent As $\ca{P}_m$ implies $\ca{P}_i$ for any $i$, $0 \le i \le m$, Lemma \ref{lem:3path} and Theorem \ref{thm:evoC4} imply $\ca{P}_i$ holds for all $i$, $0 \le i \le m$, with high probability.

\subsection{The differential equations method}\label{sec:demeth}

To show that our variables follow the conjectured trajectories, we appeal to an approach to the differential equations method presented in Lemma 7.3 from \cite{BK}.  The only difference in the statement is the notation change of $X_{j,A}^{\pm}$ instead of $Y_{j,A}^{\pm}$.  We reproduce from \cite{BK} the setup for this lemma: suppose we have a stochastic graph process defined on $[n]$, where $n$ is large.  Let $r$ be a fixed positive integer, and for $j \in [r]$, let $k_j, S_j$ be parameters (which can depend on $n$).

Suppose for each $j \in [r]$ and $A \in \bin{[n]}{k_j}$, there is a sequence of random variables $X_{j,A}(i)$, defined for $i=0,\ldots,m$ and measurable with respect to the underlying graph process.

Further, we suppose \[X_{j,A}(i+1) - X_{j,A}(i) = X^{+}_{j,A}(i) - X^{-}_{j,A}(i),\] where $X^{+}_{j,A}(i),X^{-}_{j,A}(i) \ge 0$.
We relate these sequences to functions on $[0,\infty)$ by letting $t=i/s$ for some function $s=s(n)$ that tends to infinity.  The goal is then to argue that, for some collection $x_j(t)$ of continuous functions, \[X_{j,A}(i) \approx x_j(t)S_j\] for all $j \in [r]$ and $A \in \bin{[n]}{k_j}$, $i=0\ldots,m$.  We view $1 \le j \le r$ as the type of random variable, and the set $A$ as giving its position in the graph.  The parameter $S_j$ is the size-scaling for the $j$th type of random variable.

\begin{lemma}[\cite{BK}, Lemma 7.3]\label{lem:demeth}
Let $0 < \epsilon < 1$ and $c,C> 0$ be constants, and suppose for each $j \in [r]$ we have a parameter $s_j(n)$ and functions $x_j(t),e_j(t),\theta_j(t),\gamma_j(t)$ that are smooth and nonnegative for $t \ge 0$.  For $i^* = 1,2,\ldots,m$, let $\ca{G}_{i^*}$ be the event that \[X_{j,A}(i) = \lp 1 \pm \frac{e_j(t)}{s_j} \rp\lp u_j(t) \pm \frac{\theta_j(t)}{s_j}\rp S_j\] for all $1 \le i \le i^*$, $1 \le j \le r$, and $A \in \bin{[n]}{k_j}$.  Suppose there is also a decreasing sequence of events $\ca{H}_i$, $1 \le i \le m$, such that $\lim_{n \to \infty} \pr{\ca{H}_m \mid \ca{G}_m} = 1$, and that the following conditions hold:

\begin{itemize}
    \item[1.] (Trend hypothesis) When conditioning on $\ca{G}_i \land \ca{H}_i$, we have \[\ev{X^{\pm}_{j,A}} =  \lp x^{\pm}_j(t) \pm \frac{h_j(t)}{4s_j}\rp\frac{S_j}{s},\]
        for all $j \in [r]$ and $A \in \bin{[n]}{k_j}$, where $x^{\pm}_j(t)$ and $h_j(t)$ are smooth nonnegative functions such that \[x_j'(t) = x_j^+(t) - x_j^-(t) \ \ \text{ and }\ \  h_j(t) = (e_jx_j + \gamma_j)'(t);\]
    \item[2.] (Boundedness hypothesis) For each $j \in [r]$, conditional on $\ca{G}_i \land \ca{H}_i$, we have \[X_{j,A}^{\pm}(i) < \frac{S_j}{s_j^2 k_j n^{\epsilon}};\]
    \item[3.] (Initial conditions) For all $j \in [l]$, we have $\gamma_j(0) = 0$ and $X_{j,A}(0)= S_jx_j(0)$ for all $A \in \bin{[n]}{k_j}$;
    \item[4.] We have $n^{3\epsilon} < s < m < n^2$, $m \le n^{\ep/2}s$, $s \ge 40Cs_j^2 k_j n^{\epsilon}$, $n^{2\epsilon} \le s_j < n^{-\epsilon} s$,
        \[\inf_{t \ge 0} \theta_j(t) + e_j(t)x_j(t)/2 - \gamma_j(t)/2 > c,\]
        \[\sup_{t \ge 0} |x_j^{\pm}(t)| < C, \ \ \ \ \ \sup_{t \ge 0} |x_j'(t)| < C, \ \ \ \ \ \  \int_0^{\infty} |x_j''(t)|\ dt < C,\]
        \[ \sup_{0 \le t \le m/s} |h_j(t)| < n^{\epsilon}, \ \ \ \ \ \int_0^{m/s} |h_j'(t)|\ dt < n^{\epsilon}.\]
\end{itemize}
Then $Pr[\ca{G}_m \land \ca{H}_m] \to 1$ as $n \to \infty$.\\

\end{lemma}

\subsection{Additional inequalities and the constants $\mu,\ep,V,W$}\label{sec:ineq}

As much of the remainder of this paper will be devoted to verifying the conditions of Lemma \ref{lem:demeth}, we take the opportunity now to gather a few simple inequalities.  First, in addition to the constraints on $\mu,\ep,V$ and $W$ implicit in \cite{BK}, the following bounds suffice for our application:
\[ V \ge 40,\ \ \ \ \ \ \ \ \ \ \ \ W \ge \frac{640e^{2V}}{V \log(2)}, \ \ \ \ \ \ \ \ \ \ \ \ \ \ep \le \frac{1}{100},\] and $\mu$ is chosen sufficiently small so that $e^{P(t)} \le n^{\ep/2}$ for all $t \in [0,t_{max}]$, provided $n$ is sufficiently large.

We observe that as $t_{max} = \mu(\log n)^{1/3} = o(n^{\alpha})$ for any $\alpha>0$, it follows that if $F(t)$ is a fixed polynomial and $\alpha > 0$, we have $|F(t)| \le n^{\alpha}$ on $[0,t_{max}]$ for all $n$ sufficiently large.
To simplify some of our later calculations, we mention a few additional inequalities which follow directly from our choice of the constants above, \eqref{eq:defpmtmax}-\eqref{eq:ssedef}, and \eqref{eq:kdef}, for all $t$ in $[0,t_{max}]$ and $n$ sufficiently large:
\[1 \le q(t)^{-1} \le q^2(t)e^{P(t)} \le q(t)e^{P(t)} \le e^{P(t)} \le n^{\ep},\]

\[q(t)s_e \ge n^{5 \epsilon},\ \ \ \ \ \ \ \ \ \ \ \ \text{ and }\ \ \ \ \ \ \ \ \ \ \ \ n^{1/3} \le k \le n^{1/3 + \ep}.\] Furthermore, conditioned on the event $\ca{T}_{i^*}$, $0 \le i^* \le m$, by Theorem \ref{thm:evoC4} we have \[Q(i) \ge n^{2-\ep} \ \ \ \text{ and } \ \ \ \ |C_{uv}(i)| \le n^{2/3 + \ep}\] for $0 \le i \le i^*$ and all $uv \in O(i) \cup C(i)$.

Finally, we will repeatedly make use of the following simple lemma.
\begin{lemma}
Suppose $\eta=\eta(n) \to 0$ as $n \to \infty$, and $a,b$ are positive integers.  Then, for $n$ sufficiently large,
\begin{itemize}
\item[1.] $(1 \pm a\eta)(1 \pm b\eta) \subseteq (1 \pm (a+b+1)\eta)$.
\item[2.] $(1 \pm a\eta)^{-1} \subseteq (1 \pm (a+1)\eta)$.
\end{itemize}
\end{lemma}

\begin{proof}
Both containments follow from $\eta^2 = o(\eta)$, the latter from considering the series expansion of $(1 + x)^{-1}$.
\end{proof}

\section{Proof of Lemma \ref{lem:trajectory}}\label{sec:lemma1}

Our proof of Lemma \ref{lem:trajectory} will follow from an application of Lemma \ref{lem:demeth}.  We recall $s=s(n) = n^{4/3}$, and, for $t \ge 0$, we define
\begin{align*}
x(t)& = q(t)^2/2,& x^+(t)& = 0,& x^-(t)& = c(t)q(t) = 24t^2q(t)^2,\\
y(t)& = 2tq(t),\ \ & y^+(t)& = 2q(t),\ \ \ \ \ \text{ and}& y^-(t)& = 2tc(t) = 48t^3q(t).
\end{align*} so that $x^{\pm},y^{\pm}$ are nonnegative on $[0,\infty)$, $x' = x^+-x^-$ and $y'=y^+-y^-$.

For $K \in \bin{[n]}{k}$ and $0 \le i \le m$, let $\ca{U}_{K,i}$ denote the event that $K$ is uncovered in $G(i)$, and let $\ca{E}_{K,i}$ be the event
\[\ca{E}_{K,i} = \ca{U}_{K,i} \land \ca{T}_i \land \ca{P}_i,\] where $\ca{T}_i$ and $\ca{P}_i$ are defined in Theorem \ref{thm:evoC4} and Lemma \ref{lem:3path}, respectively.  We are only interested in ensuring bounds that hold with high probability for uncovered $K$ at each step $i$, and, by Theorem \ref{thm:evoC4} and Lemma \ref{lem:3path}, $\pr{\overline{\ca{T}_m} \lor \overline{\ca{P}_m}} = o(1)$, so it suffices to show the desired bounds hold with high probability for all $K$ and $i$ for which $\ca{E}_{K,i}$ holds.  We will therefore apply Lemma \ref{lem:demeth} to a modified collection of random variables that follow the correct trajectory deterministically on the event $\overline{\ca{E}_{K,i}}$, which we define as follows.

For $i=0,1,\ldots,m$:
\begin{equation*}
X_K^+(i) = \begin{cases}
|X_K(i+1)\setminus X_K(i)| & \text{ if } \ca{E}_{K,i} \text{ holds,}\\
x^+(t)\cdot k^2n/s & \text{ otherwise,}
\end{cases}
\end{equation*}
\begin{equation*}
X_K^-(i) = \begin{cases}
|X_K(i)\setminus X_K(i+1)| & \text{ if }\ca{E}_{K,i} \text{ holds,}\\
x^-(t)\cdot k^2n/s & \text{ otherwise,}
\end{cases}
\end{equation*}
and
\begin{equation*}
\widehat{X}_K(i) = \begin{cases}
|X_K(0)|+\frac{k^3+kn-k^2}{2} & \text{ if } i=0\\
\widehat{X}_K(i-1) + X^+_K(i) - X^-_K(i) & \text{ otherwise.}
\end{cases}
\end{equation*}
\noindent Similarly, let
\begin{equation*}
Y_K^+(i) = \begin{cases}
|Y_K(i+1)\setminus Y_K(i)| & \text{ if } \ca{E}_{K,i} \text{ holds,}\\
y^+(t)\cdot k^2np/s & \text{ otherwise,}
\end{cases}
\end{equation*}
\begin{equation*}
Y_K^-(i) = \begin{cases}
|Y_K(i)\setminus Y_K(i+1)| & \text{ if } \ca{E}_{K,i} \text{ holds,}\\
y^-(t)\cdot k^2np/s & \text{ otherwise,}
\end{cases}
\end{equation*}
and
\begin{equation*}
\widehat{Y}_K(i) = \begin{cases}
|Y_K(0)| & \text{ if } i=0,\\
\widehat{Y}_K(i-1) + Y^+_K(i) - Y^-_K(i) & \text{ otherwise.}
\end{cases}
\end{equation*}
\noindent It follows that on the event $\ca{E}_{K,i}$, $\widehat{X}_K(i) = |X_K(i)|+\frac{k^3 + kn-k^2}{2} \approx |X_K(i)|$ and $\widehat{Y}_K(i) = |Y_K(i)|$.\\

To set up our application, we recall $m=m(n) = \mu (\log n)^{1/3}n^{4/3}$, and we let $c = 1/4$ and take $C>0$ to be a sufficiently large constant.  We let $k_1 = k_2 = k$, $x_1=x$, $x_2 = y$, and for $K \in \bin{[n]}{k}$, we let \[X_{1,K}(i) = \widehat{X}_K(i), \ \ \ \ \ \ S_1 = k^2n,\ \ \ \ \ \ \ X_{2,K}(i) = \widehat{Y}_K(i),\ \ \ \text{ and } S_2 = k^2np.\] As $\widehat{X}_K^{\pm} = X_K^{\pm}$ and $\widehat{Y}_K^{\pm} = Y_K^{\pm}$, we will write the latter for ease of reading.

We define, for $t \ge 0$,
\begin{equation*}
\gamma(t) = \frac{1}{4}\lp 1 - \exp(-640e^{2V}\cdot t)\rp \ \ \ \ \ \ \text{and}\ \ \ \ \ \ \ \ \theta(t) = \frac{1}{2} + \gamma(t).
\end{equation*}
For $j \in \{1,2\}$ we define the remaining error parameters as
\[s_j = n^{3\ep},\ \ \ \ \ e_j(t) = e(t), \ \ \ \ \ \ \ \gamma_j(t) = \gamma(t)\ \ \ \ \ \ \text{ and }\ \ \ \ \ \theta_j(t) = \theta(t),\] where $e(t)$ is defined in \eqref{eq:Pedef}.

So, for $0 \le i^* \le m$, $\ca{G}_{i^*}$ is the event that
\begin{equation} \label{eq:good}
\begin{split}
\widehat{X}_K(i) &= \lp 1 \pm \frac{e(t)}{s_1} \rp \lp x(t) \pm \frac{\theta(t)}{s_1} \rp k^2n \text{ and }\\
\widehat{Y}_K(i) &= \lp 1 \pm \frac{e(t)}{s_2} \rp \lp y(t) \pm \frac{\theta(t)}{s_2} \rp k^2np
\end{split}
\end{equation}
for all $K \in \bin{[n]}{k}$ and $0 \le i \le i^*$.  We take the event $\ca{H}_i = \ca{G}_i$ for all $i$, which trivially is decreasing and satisfies $\lim_{n \to \infty} Pr(\ca{H}_m | \ca{G}_m) = 1$.

The initial conditions follow easily: $\widehat{Y}_K(0) = 0 = y(0)k^2np$, and \[\widehat{X}_K(0) = \bin{k}{2}(n-k) + \frac{k^3 + kn-k^2}{2} = \frac{k^2n}{2} = x(0)k^2n.\]
We point out that as $\theta(t) \le 3/4$ and $(k^3 + kn-k^2)/2 = o(1) \cdot k^2n/s_1$, the conclusions of Lemma  \ref{lem:trajectory} hold on the event $\ca{G}_m \land \ca{T}_m \land \ca{P}_m$, so as $\ca{T}_m \land \ca{P}_m$ holds with high probability, it suffices to show $\ca{G}_m$ holds with high probability.

We next note that, as intended, the trend and boundedness hypotheses follow deterministically for $X_K^{\pm}$ and $Y_K^{\pm}$ on the event $\overline{\ca{E}_{K,i}}$ - the trend hypothesis is trivial.  The boundedness hypothesis follows from the inequalities (which we will establish!) $|x^{\pm}| \le C, |y^{\pm}| \le C$ and $s \ge 40Cs_j^2k_jn^{\ep}$ for $j \in \{1,2\}$.  It therefore remains to show they hold when conditioned on $\ca{G}_i \land \ca{E}_{K,i}$.

\subsection{Open triples}\label{sec:open}
\subsubsection{Trend hypothesis}\label{sec:opentrend}

As $X_K^+(i)=0$ for all $i$, the trend hypothesis for $X_K^+$ follows, so we turn to $X_K^-$.  To simplify our calculations, all functions in the expressions which follow are assumed to be evaluated at $t=t(i)$, and we will write $q$ in place of $q(t)$, etc..  To avoid potential confusion, we will use ``$e$'' to refer to the function defined in \eqref{eq:Pedef}, and ``$\es$" to refer to the constant $\es = 2.718\ldots$.

Conditioned on $\ca{E}_{K,i}$, a triple $(uv,w) \in X_K(i)$ gets counted by $X_K^-(i)$ if and only if $e_{i+1} \in \{uw,vw\} \cup C_{uw}(i) \cup C_{vw}(i)$.  As $K$ is uncovered, it follows that $uw \notin C_{vw}(i)$ and vice-versa, and therefore the probability of this occurring, conditioned on $\ca{G}_i \land \ca{E}_{K,i}$, is

\begin{eqnarray*}
\frac{|C_{uw}(i) \cup C_{vw}(i)| + 2}{Q(i)} &=& \frac{2(1 \pm e/s_e)(c \pm 12/s_e)p^{-1}/2 \pm (n^{-1/4}p^{-1} + 2)}{(1 \pm e/s_e)(q \pm 1/s_e)n^2/2}\\
&\subseteq& \frac{2(1 \pm e/s_e)(c \pm (12/s_e + n^{-1/4} + 2p))p^{-1}/2}{(1 \pm e/s_e)q(1 \pm 1/(qs_e))n^2/2}\\
&\subseteq& \frac{2(1 \pm e/s_e)(c \pm 13/s_e)}{(1 \pm e/s_e)q(1 \pm 1/(qs_e))}\cdot \frac{p^{-1}}{n^2}\\
&\subseteq& 2 \lp 1 \pm \frac{e}{s_e}\rp \lp 1 \pm \frac{2e}{s_e}\rp \lp \frac{c}{q} \pm \frac{13}{qs_e}\rp \lp 1 \pm \frac{2}{qs_e} \rp \cdot \frac{p^{-1}}{n^2}\\
&\subseteq& 2 \lp 1 \pm \frac{4e}{s_e} \rp \lp \frac{c}{q} \pm \frac{13 + 2c/q + 26/(qs_e)}{qs_e} \rp \cdot \frac{p^{-1}}{n^2}\\
&\subseteq& \lp 1 \pm \frac{4e}{s_1} \rp \lp \frac{2c}{q} \pm \frac{1}{s_1} \rp \cdot \frac{p^{-1}}{n^2},
\end{eqnarray*} where the last containment follows from $c/q = 24t^2 \le n^{\ep}$ and $qs_e \ge n^{2\ep}s_1$. Summing this over all $(uv,w) \in X_K(i)$ and using \eqref{eq:good} yields
\begin{eqnarray*}
\ev{X_K^-(i) | \ca{G}_i \land \ca{E}_{K,i}} &=& \lp 1 \pm \frac{e}{s_1}\rp \lp x \pm \frac{\theta}{s_1} \rp k^2n \cdot \lp 1 \pm \frac{4e}{s_1} \rp \lp \frac{2c}{q} \pm \frac{1}{s_1} \rp \cdot \frac{p^{-1}}{n^2}\\
&\subseteq& \lp 1 \pm \frac{6e}{s_1}\rp \lp \frac{2xc}{q} \pm \frac{x + 2c/q + 1/s_1}{s_1} \rp \cdot \frac{k^2n}{s}\\
&\subseteq& \lp 1 \pm \frac{6e}{s_1}\rp \lp \frac{2xc}{q} \pm \frac{48t^2 + 2}{s_1} \rp \cdot \frac{k^2n}{s}\\
&\subseteq& \lp \frac{2xc}{q} \pm \frac{12exc/q + 48t^2 + 2 + 6e(48t^2+2)/s_1}{s_1} \rp \frac{k^2n}{s}\\
&\subseteq& \lp \frac{2xc}{q} \pm \frac{288t^2xe + 48t^2 + 3}{s_1} \rp \frac{k^2n}{s}.
\end{eqnarray*}

It  remains to show that $288t^2x(t)e(t) + 48t^2 + 3 \le h_1(t)/4$, where $h_1(t) = (xe + \gamma)'(t)$: routine calculations yield
\begin{eqnarray*}
h_1 &=& (-48t^2x)e + xP'\es^{P(t)} + \gamma'\\
&\ge& -48t^2x\es^{P(t)} + W(3t^2 + 1)x\es^{P(t)}\\
&\ge& W(t^2+1)x\es^{P(t)},
\end{eqnarray*}
as $48 \le 2W$.  As the inequalities $x(t)e(t) \le x(t)\es^{P(t)}$ and $1 \le 2x(t)\es^{P(t)}$ hold, it suffices to show \[(384t^2 + 6)x(t)\es^{P(t)} \le \frac{W}{4}(t^2+1)x(t)\es^{P(t)},\] which follows as $W \ge 4\cdot 384 = 1544$.

\subsubsection{Boundedness hypothesis}\label{sec:openbd}

As no new open triples are created in any step, the bound follows for $X_K^+$ trivially.  Furthermore, as every open pair in $O(i)$ lies in at most $k$ triples in $X_K(i)$, conditioned on $\ca{G}_i \land \ca{E}_{K,i}$ we have \[|X_K^-(i)| \le (|C_{e_{i+1}}(i)| + 1)\cdot k \le n^{2\ep}\cdot n < \frac{kn}{n^{7\ep}} = \frac{S_1}{s_1^2kn^{\ep}},\] as $k \ge n^{1/3} > n^{9\ep}$ for $n$ sufficiently large.

\subsection{Partial triples}\label{sec:part}

\subsubsection{Trend hypothesis}\label{sec:parttrend}

We begin by establishing the bounds for $\ev{Y_K^+ | \ca{G}_i \land \ca{E}_{K,i}}$.  A triple $(uv,w) \in X_K(i)$ enters $Y_K(i+1)$ if and only if $e_{i+1} \in \{uw,vw\}$, which occurs with probability \[\frac{2}{Q(i)} = \frac{2}{(1 \pm e/s_e)(q \pm 1/s_e)n^2/2} = \frac{2}{q} \cdot \lp 1 \pm \frac{2e}{s_e} \rp \lp 1 \pm \frac{2}{qs_e} \rp \cdot \frac{2}{n^2}.\] Summing over the triples $(uv,w) \in X_K(i)$ and using \eqref{eq:good} yields
\begin{eqnarray*}
\ev{Y^+_K(i) | \ca{G}_i \land \ca{E}_{K,i}} &=& \lp 1 \pm \frac{e}{s_1} \rp \lp x \pm \frac{\theta}{s_1} \rp k^2n \cdot \frac{2}{q} \cdot \lp 1 \pm \frac{2e}{s_e} \rp \lp 1 \pm \frac{2}{qs_e} \rp \cdot \frac{2}{n^2}\\
&\subseteq& 4 \lp 1 \pm \frac{4e}{s_1} \rp \lp \frac{x}{q} \pm \frac{q^{-1}}{s_1} \rp  \lp 1 \pm \frac{2}{qs_e} \rp  \cdot  \frac{k^2}{n}\\
&\subseteq& 4 \lp 1 \pm \frac{4e}{s_1} \rp \lp \frac{x}{q} \pm \frac{q^{-1}}{s_1} \rp  \lp 1 \pm \frac{1}{s_1} \rp  \cdot  \frac{k^2np}{s}\\
&\subseteq& 4 \lp 1 \pm \frac{4e}{s_1} \rp \lp \frac{x}{q} \pm \frac{q^{-1}+x/q+q^{-1}/s_1}{s_1} \rp  \cdot  \frac{k^2np}{s}\\
&\subseteq& 4 \lp 1 \pm \frac{4e}{s_1} \rp \lp \frac{x}{q} \pm \frac{q^{-1}+1}{s_1} \rp  \cdot  \frac{k^2np}{s}\\
&\subseteq& 4\lp \frac{x}{q} \pm \frac{4exq^{-1} + q^{-1} + 1 + 4e(q^{-1}+1)/s_1}{s_1} \rp \cdot  \frac{k^2np}{s}\\
&\subseteq& \lp \frac{4x}{q} \pm \frac{8qe + 4q^{-1} + 8}{s_1} \rp \cdot  \frac{k^2np}{s},
\end{eqnarray*} where the last containment follows from the inequality $4e(q^{-1}+1) < 8n^{2\ep} < s_1$ for $n$ sufficiently large.

It remains to show that $8q(t)e(t) + 4q(t)^{-1} + 8 < h_2(t)/4$, where $h_2(t)=(ye+\gamma)'(t)$.  Straightforward calculations yield \begin{eqnarray}
h_2 &=& \nonumber (2-48t^3)qe + (2tq)W(3t^2 + 1)\es^{P(t)} + \gamma'\\
&\ge& \label{eq:h2lowerbd} 2W(t^3+t)q\es^{P(t)}+ \gamma'.
\end{eqnarray}  To show the desired inequality, we consider two cases: $t < V/W$ and $t \ge V/W$.
If $t < V/W <1$, then as $t^3+t \le 2t$, \[\es^{P(t)} \le \es^{W(2V/W)}= \es^{2V}.\] As $q(t)\es^{P(t)} \ge q(t)^{-1} \ge 1$, it follows that $8q(t)e(t) + 4q(t)^{-1} + 8 \le 20q(t)\es^{P(t)} \le 20\es^{2V}$.  On the other hand, we have \[\frac{h_2(t)}{4} \ge \frac{\gamma'(t)}{4} = \frac{640\es^{2V}\cdot \es^{-640\es^{2V}t}/4}{4} \ge \frac{640\es^{2V}}{32} = 20\es^{2V},\] as $W \ge 640\es^{2V}\cdot V/\log(2)$.
If $t \ge V/W$, then, as $\gamma'$ is nonnegative, \[\frac{h_2(t)}{4} \ge \frac{W}{2}\cdot t q(t)\es^{P(t)} \ge \frac{V}{2} q(t)\es^{P(t)},\] which suffices as $V \ge 40$.\\

Next, we turn to $\ev{Y^-_K|\ca{G}_i \land \ca{E}_{K,i}}$: for each triple $(uv,w) \in Y_K(i)$, where without loss of generality $uw \in O(i)$, the probability that $(uv,w)$ gets counted by $Y^-_K(i)$ conditioned on $\ca{G}_i \land \ca{E}_{K,i}$ is
\begin{eqnarray*}
\frac{|C_{uw}(i)| + 1}{Q(i)} &=& \frac{(1 \pm e/s_e)(c \pm 12/s_e)p^{-1}/2+1}{(1 \pm e/s_e)(q \pm 1/s_e)n^2/2}\\
&\subseteq& \lp 1 \pm \frac{4e}{s_e} \rp \lp \frac{c}{q} \pm \frac{2c/q + 13 + 26/(s_eq)}{s_eq}\rp \frac{p^{-1}}{n^2}\\
&\subseteq& \lp 1 \pm \frac{4e}{s_2} \rp \lp \frac{c}{q} \pm \frac{1}{s_2}\rp \frac{p^{-1}}{n^2},
\end{eqnarray*} the last containment following from $s_eq \ge n^{\ep}s_2$ and $2c/q = 48t^2 \le n^{\ep}$ on $[0,t_{max}]$.

Consequently,
\begin{eqnarray*}
\ev{Y_K^-(i) | \ca{G}_i \land \ca{E}_{K,i}} &=& \lp 1 \pm \frac{e}{s_2}\rp \lp y \pm \frac{\theta}{s_2} \rp k^2np \cdot \lp 1 \pm \frac{4e}{s_2} \rp \lp \frac{c}{q} \pm \frac{1}{s_2}\rp \frac{p^{-1}}{n^2}\\
&\subseteq& \lp 1 \pm \frac{6e}{s_2}\rp \lp \frac{yc}{q} \pm \frac{y + c/q + 1/s_2}{s_2}\rp \frac{k^2np}{s}\\
&\subseteq& \lp 1 \pm \frac{6e}{s_2}\rp \lp \frac{yc}{q} \pm \frac{26t^2 + 1}{s_2}\rp \frac{k^2np}{s}\\
&\subseteq& \lp \frac{yc}{q} \pm \frac{6eyc/q + 26t^2+1 + 6e(26t^2+1)/s_2}{s_2} \rp \frac{k^2np}{s}\\
&\subseteq& \lp 48t^3q(t) \pm \frac{288t^3qe + 26t^2+2 }{s_2} \rp \frac{k^2np}{s}.
\end{eqnarray*}

To establish the required bound, we first observe that, using $t^2 \le t^3 + t$ and $q(t)\es^{P(t)} \ge 1$ for $t \ge 0$, we have \[288t^3q(t)e(t) + 26t^2 + 2 \le 314(t^3+t)q(t)\es^{P(t)}+2.\]
As $W \ge 4(314) = 1256$, we have $314(t^3+t)q(t)\es^{P(t)} \le \frac{W}{4}(t^3+t)q(t)\es^{P(t)}$, so by \eqref{eq:h2lowerbd} it suffices to show that \[2 \le \frac{W}{4}(t^3+t)q(t)\es^{P(t)} + \frac{\gamma'(t)}{4}.\]  Again considering the cases $t < V/W$ and $t \ge V/W$ separately, by the arguments given above we have $\gamma'(t)/4 \ge 20e^{2V}$ for $t \le V/W$ and $\frac{W}{4}(t^3+t)q(t)\es^{P(t)} \ge V/4$ for $t \ge V/W$, which suffices as $V \ge 8$.

\subsubsection{Boundedness hypothesis}\label{sec:partbd}

We recall that it suffices to show the boundedness hypothesis holds conditioned on $\ca{G}_i \land \ca{E}_{K,i} = \ca{G}_i \land \ca{U}_{K,i} \land \ca{T}_i \land \ca{P}_i$, which we assume throughout this subsection.  We start with $Y_K^+$: as a given open pair lies in at most $k$ open triples, we have \[|Y_K^+(i)| \le k < k\cdot \frac{n^{1/3}}{n^{7\ep}} = \frac{S_2}{s_2^2kn^{\ep}}.\]

Turning to $Y_K^-$, we recall that as $K$ is uncovered, each partial triple $(uv,w) \in Y_K(i)$ contains a unique open pair, which we will take to be $uw$ without loss of generality.  However, the trivial bound $|Y_K^-(i)| \le |C_{e_{i+1}}(i)| + 1$ does not suffice, as for most steps $i$, $|C_{e_{i+1}}(i)| = n^{2/3+o(1)}$, while the required upper bound $S_2/(s_2^2kn^{\ep}) = n^{2/3-7\ep+o(1)}$, so we must be more careful.

A triple $(uv,w) \in Y_K(i)$ is counted by $Y_K^-(i)$ if and only if $e_{i+1} \in \{uw\} \cup C_{uw}(i)$.  Suppose $e_{i+1} = xy$: we separately bound the number of partial triples $(uv,w)$ removed from $Y_K(i)$ based on the intersection of the open pair $uw$ with $xy$: trivially, at most one such triple has $uw=xy$, so let
\begin{eqnarray*}
A_{1} &=& \{(uv,w) \in Y_K(i): uw \in O(i) \cap C_{xy}(i),\  uw \cap xy= \emptyset\},\\
A_{2} &=& \{(uv,w) \in Y_K(i): uw \in O(i) \cap C_{xy}(i),\ uw \cap xy = \{w\}\}, \text{ and}\\
A_{3} &=& \{(uv,w) \in Y_K(i): uw \in O(i) \cap C_{xy}(i),\  uw \cap xy = \{u\}\}.
\end{eqnarray*}

Suppose first that $(uv,w) \in A_1$: as $uw \in C_{xy}(i)$, we have that either $\{ux,wy\} \subseteq E(i)$ or $\{uy,wx\} \subseteq E(i)$: in the former case, as $K$ is uncovered we must have $u$ as the unique neighbor of $x$ in $K$ and $w$ as a neighbor of $y$, for which there at most $\Delta(G(i))$ such choices.  Analogous reasoning for the second case yields
\begin{equation}\label{eq:A1bd}
|A_1| \le 2\Delta(G(i)) \le 8\mu (n \log n)^{1/3}.
\end{equation}

Next, consider a $(uv,w) \in A_2$, and suppose first that $x = w$: then there is a vertex $z$ such that $\{yz,zu\} \subseteq E(i)$.  There are at most $\Delta(G(i))$ choices of $z$ adjacent to $y$, and fixing $z$, at most one choice of $u \in K$; analogous reasoning for the case $y=w$ yields
\begin{equation}\label{eq:A2bd}
|A_2| \le 8\mu (n \log n)^{1/3}.
\end{equation}

Turning now to $A_3$, we partition $A_3$ into $A_3' \cup A_3''$, where $A_3'$ contains those triples $(uv,w) \in A_3$ with $u = x$, and $A_3''$ contains those triples with $u = y$. We first bound $|A_3'|$: let \[B = \{w \in [n]: (uv,w) \in A_3' \text{ for some } v \in K\}.\]  As for all $(uv,w) \in A_3'$ we have $u=x$ and as no open pair is contained in more than one partial triple, it follows that $|B| = |A_3'|$.  Since for each $w \in B$ we also have $uw \in C_{xy}(i)$, it follows that there exists a vertex $z_w$ such that $\{yz_w,z_ww\} \subseteq E(i)$. Let \[B' = \{w \in B: z_w \notin K, yw \notin E(i)\}.\]

\begin{claim}\label{clm:B'bd}
\[|B| - 2\Delta(G(i)) \le |B'| \le kn^{1/4}.\]
\end{claim}

\begin{proof}
For the lower bound, it follows trivially that $y$ has at most  $\Delta(G(i))$ neighbors in $B$.  Similarly, if $w,w' \in B$ with $z_w,z_{w'} \in K$, then, as $z_w,z_{w'}$ are neighbors of $y$ and $K$ is uncovered, $z_{w}=z_{w'}$, so we have $z_w \in K$ for at most $\Delta(G(i))$ distinct $w \in B$.

For the upper bound, we note that each $w \in B'$ has exactly one neighbor in $K \setminus \{y\}$, and that $z_w \notin K$.  It follows that there are at least $|B'|$ paths of length $3$ in $G(i)$ from $y$ to $K \setminus \{y\}$, so for some $v \in K\setminus \{y\}$, there are at least $|B'|/k$ paths of length $3$ from $y$ to $v$.  But as $\ca{P}_i$ holds, there are also at most $n^{1/4}$ such paths, and the result follows.
\end{proof}

\noindent From Claim \ref{clm:B'bd}, \[|A_3'| = |B| \le |B'| + 2\Delta(G(i)) \le kn^{1/4} + 8\mu(n \log(n))^{1/3},\] and applying analogous arguments to $|A_3''|$ lets us conclude
\begin{equation}\label{eq:A3bd}
|A_3| \le 2kn^{1/4} + 16\mu(n \log(n))^{1/3}.
\end{equation}

Finally, combining \eqref{eq:A1bd}-\eqref{eq:A3bd} with the fact that $(n \log n)^{1/3} = o(kn^{1/4})$, we have
\[Y_K^-(i) \le 1 + |A_1| + |A_2| + |A_3| \le 2kn^{1/4} + o(kn^{1/4}) \le 3kn^{1/4}\] for $n$ sufficiently large.
As $k \ge n^{1/3}$ and $3k \le n^{1/3 + \ep}$ for $n$ sufficiently large, therefore
\[Y_K^-(i) \le n^{7/12 + \ep} \le n^{8/12-7\ep} = \frac{n^{2/3}}{n^{7\ep}} \le \frac{S_2}{s_2^2kn^{\ep}},\] provided $\ep \le 1/96$, and the boundedness hypothesis is verified.

\subsection{Analytic considerations}\label{sec:analytic}

Here we verify the remaining inequalities from Part 4. of Lemma \ref{lem:demeth}, recalling that we chose $c = 1/4$ and $C$ sufficiently large.  First, from \eqref{eq:defpmtmax}, we have $t_{\max} = \mu(\log n)^{1/3} = o(n^{\ep/2})$ so for large $n$, we have $n^{3\ep} < s < t_{max}s=m < n^{\ep/2}s \le n^2$.  For any fixed constant $C>0$, $j \in \{1,2\}$, and $n$ sufficiently large, as $k_j \le n^{1/3+\ep}$,  \[40Cs_j^2k_jn^{\ep} \le 40Cn^{1/3 + 8\ep} < n^{1/3 + 9\ep} < s.\]

As $x(t),y(t),e(t),\gamma(t)$ are nonnegative, it follows from the definitions that, for $j \in \{1,2\}$,\[\inf_{t \ge 0} \theta_j(t) + \frac{e_j(t)x_j(t)}{2} - \frac{\gamma_j(t)}{2} \ge \inf_{t \ge 0} \frac{1}{2} + \frac{\gamma(t)}{2} > \frac{1}{4}.\]

Next, we observe from the definitions of $x,x^{\pm},y,y^{\pm}$ and straightforward differentiation that we can bound, for $t \ge 0$, $|x'|$, $|y'|$, $|x^{\pm}|$, $|y^{\pm}|$, $|x''|$, and $|y''|$ above by a function of the form $H(t)=F(t)\es^{-8t^3}$, where $F$ is a polynomial of degree $5$ with nonnegative coefficients.  It is straightforward to see that both $\sup_{t \ge 0} H(t)$ and $\int_0^{\infty} H(t)$ are finite, so provided $C$ is greater than the larger of these two, $\sup_{t \ge 0} |x_j^{\pm}(t)|$, $\sup_{t \ge 0} |x_j'(t)|$, and $\int_{0}^{\infty} |x_j''(t)|\ dt$ are all less than $C$.

Finally, turning to the inequalities involving $h_j = (e_jx_j + \gamma_j)'$, it is
easy to see that $\sup_{t \ge 0} |\gamma'(t)|$ and $\int_0^{\infty} |\gamma''(t)|$ are bounded.  Calculations similar to those establishing lower bounds on $h_1,h_2$ easily yield that $|(e_jx_j)'|$ and $|(e_jx_j)''|$ are bounded above by a function of the form $F(t)\es^{P(t)}$, where $F$ is a polynomial of degree $5$ with nonnegative coefficients.  As $F(t)$ and $\es^{P(t)}$ are increasing, $\es^{P(t)} < (\es^{P(t)})'$, and $m/s=t_{max}$, it follows that \[\int_0^{t_{max}} F(t)\es^{P(t)}\ dt \le F(t_{max})\es^{P(t_{max})} = \sup_{0 \le t \le t_{max}} F(t)\es^{P(t)}.\]  As $t_{max}=\mu (\log n)^{1/3} = o(n^{\ep/10})$, say, $F(t_{max}) = o(n^{\ep/2})$, and as $\es^{P(t)} \le n^{\ep/2}$ on $[0,t_{max}]$, the bounds $\sup_{0 \le t \le m/s} |h_j(t)| < n^{\ep}$ and $\int_0^{m/s} |h_j'(t)|\ dt < n^{\ep}$ easily follow, and the proof of Lemma \ref{lem:trajectory} is complete.

\end{document}